\documentclass[11pt]{amsart}
\usepackage{amsmath}
\usepackage{amsfonts}
\usepackage{amsthm}
\usepackage{amssymb}
\numberwithin{equation}{section}

\newtheorem{theorem}{Theorem}

\newtheorem{proposition}[theorem]{Proposition}
\newtheorem{lemma}[theorem]{Lemma}

\newtheorem{remark}[theorem]{Remark}

\begin{document}

\title[Root System]{Root System of a Perturbation of a Selfadjoint Operator with Discrete Spectrum}

\author{James Adduci}

\address{Department of Mathematics,
The Ohio State University,
 231 West 18th Ave,
Columbus, OH 43210, USA}

\email{adducij@math.ohio-state.edu}

\author{Boris Mityagin}

\address{Department of Mathematics,
The Ohio State University,
 231 West 18th Ave,
Columbus, OH 43210, USA}

\email{mityagin.1@osu.edu}


\begin{abstract}
We analyze the perturbations $T+B$ of a selfadjoint operator $T$  in a Hilbert space $H$ with discrete spectrum 
$\{ t_k \}$ , $T \phi_k = t_k \phi_k$, as an extension of our constructions in \cite{admt2} where $T$ was a harmonic
oscillator operator. In particular, if $t_{k+1}-t_k \geq c k^{\alpha - 1}, \quad \alpha > 1/2$ and 
$\| B \phi_k \| = o(k^{\alpha - 1})$ then the system of root vectors of $T+B$, eventually eigenvectors of geometric multiplicity $1$, 
is an unconditional basis in $H$.
\end{abstract}

\maketitle

\section{Statement of main results}
Let $H$ be a separable Hilbert space.
Consider an operator $T$ with domain $\text{dom} T$ whose spectrum consists of a countable set of eigenvalues $\tau = \{t_k\}_{k=1}^{\infty}$
with corresponding eigenvectors $\{ \phi_k \}$,
\begin{align*}
T \phi_k = t_k \phi_k,
\end{align*}
which form an orthonormal basis in $H$. Let us also assume that $t_{k+1} - t_k > 0$ and that
for some fixed $p \in \mathbb{Z}_+$, $d>0$
\begin{align} \label{1_0615} 
t_{k+p} - t_k > d \quad \forall k \in \mathbb{Z}_+.
\end{align}
Define $\triangle t_k = t_{k+1}-t_k$. Then (\ref{1_0615}) says
$\triangle t_k + \triangle t_{k+1} + \ldots + \triangle t_{k+p-1} >d \quad \forall k$.
Hence, for any $k \in \mathbb{Z}_+$, there exists $ \gamma(k) \in \{ 0, 1, \ldots, p-1 \}$ such that 
\begin{enumerate}
\item $\triangle t_{k + \gamma(k)} \geq d/p$ and  
\item $\triangle t_{k+j} < d/p \quad \forall j < \gamma(k)$.
\end{enumerate}
Let $j_1 = 1$ and $j_k = j_{k-1} + \gamma(j_{k-1})$ for $k > 1$
and define $T_k = t_{j_k}$.
Define the intervals
\begin{align*}
F_1 = [T_1 - \frac{d}{2p}, T_2 + \frac{d}{2p}], \quad F_k = [T_k + \frac{d}{2p}, T_{k+1}+\frac{d}{2p}], \quad k> 1.
\end{align*}
It follows that 
\begin{align*}
\tau \subset \cup_{k=1}^{\infty} F_k \quad \text{and} \quad \# (\tau \cap F_k) \leq p \quad \forall k.
\end{align*}
Set 
\begin{align} 
\label{0705_1}
\Pi_k = \{ a + ib : a \in F_k, |b| \leq \frac{d}{2p} \}, \quad 
\Gamma_k = \partial \Pi_k
\end{align}
and for $z \notin \text{Sp}T$,
\begin{align}
\label{0705_3}
R^0(z) = (z-T)^{-1}.
\end{align} 
With 
\begin{align*}
P_k^0 = \frac{1}{2 \pi i} \int_{\Gamma_k} R^0(z) dz
\end{align*}
we have a resolution of the identity $\sum_{k=1}^{\infty} P_k^0$.

Consider the perturbed operator $L = T + B$ with $B$ closed and $\text{dom} B \supseteq   \text{dom}T  $.
Set $\beta = \{ \beta_k = \| B \phi_k \|^2 \}$. In Proposition \ref{A} we will use the condition 
\begin{align} \label{1}
\lim \sup \beta_k < \left( \frac{d}{2p} \right)^2 \left( \frac{1}{8p(1+ \pi^2/3)} \right).
\end{align}
This condition implies the existence of integers $M, N$ such that 
\begin{align} \label{2}
\beta_k & \leq  \left( \frac{d}{2p}\right)^2 \left( \frac{1}{8p(1+ \pi^2/3)} \right) \forall  k \geq M, \\
\label{3}
\| \beta \|_{\infty} &\leq \frac{d^2}{16 p^2} \left(2p \sum_{j=N+1}^{\infty} 1/j^2 \right)^{-1}.
\end{align}
Define $h$ to be a positive constant which satisfies 
\begin{align*}
\sum_{j=0}^{\infty} \frac{1}{h^2 + \left( \frac{jd}{2p} \right)^2} \leq	 \frac{1}{8 p \| \beta \|_{\infty}}
\end{align*}
and set 
\begin{align*}
\Pi_0 &= \{a + ib:  -h \leq  a \leq T_{M+N+1} + \frac{d}{2p}, |b| \leq h \}, \quad \Gamma_0 = \partial \Pi_0, \\
R(z)  &= (z-L)^{-1} \quad \forall z \notin \text{Sp} L.
\end{align*}
\begin{proposition} \label{A}
Suppose the conditions (\ref{1_0615}) and (\ref{1}) hold and that $M, N$ satisfy (\ref{2} - \ref{3}); $K=M+N$. Then, with 
the notation  (\ref{0705_1}) - (\ref{0705_3}), $\text{Sp} L $ is discrete and contained in 
$\Pi_0 \cup \cup_{j=K+1}^{\infty} \Pi_j$. 
\end{proposition}
This proposition implies that the following operators are well-defined
\begin{align*}
S_K &= \frac{1}{2 \pi i} \int_{\partial \Gamma_0} R(z) dz, \\ 
P_k &= \frac{1}{2 \pi i} \int_{\partial \Gamma_k} R(z) dz \quad \text{for} \,\, k \geq K+1.
\end{align*}
\begin{proposition} \label{0720_A}
Under the conditions of Proposition \ref{A}, 
\begin{align} 
\label{0722_1}
\text{dim} S_K &= \sum_{j=1}^{K} \text{dim} P_j^0 \leq pK, \\ 
\label{0722_2}
\text{dim} P_j &= \text{dim} P_j^0 \leq p \quad \text{ for all } j \geq K+1  \quad \text{and} \\
\label{0720_2}
\| R(z) \|^2 &\leq \left( \frac{d}{p} \right)^2 \quad \forall z \notin  \Pi_0 \cup \cup_{j=K+1}^{\infty} \Pi_j.
\end{align}
\end{proposition}
\begin{theorem} \label{B}
Suppose the condition (\ref{1_0615}) holds and 
$\| B \phi_k \| \rightarrow 0$ as $k \rightarrow \infty$. Then there is a bounded operator $W$  such that 
$W P_k W^{-1}  = P_k^0$, $\text{dim} P_k^0 \leq p$ for all $k >K$ and $W S_K W^{-1}   =  \sum_{k=1}^{K}  P_k^0 $.
Hence, $\{S_K, P_{K+1}, P_{K+2}, \ldots \}$ is a Riesz system of projections.
\end{theorem}
Basically this statement is proven in \cite[Thm. 2]{shkalikov} where the condition (\ref{1}) is weaker (see (1.2) there)
but the dimension of the projectors $\{ P_k \}$ in the Riesz system are bounded by $2p$, not by $p$. Our alternative approach 
--as in \cite{admt2}-- is based on the boundedness of the discrete Hilbert transform and its adjustments. 

We will also consider the case in which the sequence of eigenvalues satisfies the growth condition 
\begin{align} \label{0704_1}
t_{k+1} - t_k \geq \kappa k^{\alpha -1} \quad \forall k \in \mathbb{N}
\end{align}
where $\alpha \in (0,\infty) \backslash \{ 1 \}.$  
\\
Define
\begin{align} \label{0704_2}
v = 2^{\frac{1}{|\alpha - 1|}}
\end{align}
and put 
\begin{align*} 
V_0 = [0,v) \cap \mathbb{N} , \quad V_k = [v^k, v^{k+1}) \cap \mathbb{N} \quad \forall k \in \mathbb{N}.
\end{align*}
Consider a closed operator $B$ with $\text{dom} B \supseteq \text{dom} T$ and 
\begin{align}  \label{0704_4}
\| B \phi_k \| = c_k k^{\alpha -1} 
\quad \text{with} \quad \lim_{k \rightarrow \infty} c_k =0.
\end{align}
(See the remark in Section \ref{0404_2}).
Set $L = T + B$ and $c_{\infty} = \sup | c_k|$ .

For each $ k \in \mathbb{N}$ define
\begin{align} \label{0704_5}
\Pi_k = \{ a+i b : t_k - (\kappa/2)(k-1)^{\alpha-1} \leq a \leq t_k + (\kappa /2)k^{\alpha -1}, \quad |b| \leq (\kappa/2) k^{\alpha -1} \}
\end{align}
and $\Lambda_k = \partial \Pi_k$ so that
\begin{align}
\label{0722_3}
| \Lambda_k | &\leq 4 \kappa k^{\alpha - 1} \quad \text{and} \\
\notag
\{t_k \} &\subset \cup_{1}^{\infty} \Pi_j
\end{align}

Now select $N$ large enough so that 
\begin{align}
\label{0704_6}
v^{N \alpha} &> c_\infty^2 \left( \frac{1}{1-\frac{1}{v}} \right) 
\quad \text{and}\\
\label{0704_7}
c_j^2 &\leq (1/4) \left( \frac{16}{\kappa^2}\left( 1+ \frac{2 \pi^2}{3} \right)  + \frac{4}{1-\frac{1}{v}}      \right)^{-1} \quad 
\forall j \in V_j, \quad j \geq N/2.
\end{align}
Finally set 
\begin{align} \notag
\ell &= \sup \left\{ \cup_{j \leq N}  V_j \right\}, \\
\label{0704_8}
Y &= \left( 4 c_{\infty}^2 v \sum_{j=1}^{N} \left( \sqrt{v}/2\right)^{2j} \right)^{1/2} \quad \text{and} \\
\label{0704_9}
\Pi_0 &= \{ a + ib : -Y \leq a \leq t_\ell + ( \kappa /2) \ell^{\alpha -1}, \quad |b| \leq Y\}
\end{align}
and define 
\begin{align} \label{0705_4}
R^0(z) &= (z - T)^{-1}, \quad R(z) = (z-L)^{-1}, \\
\label{0705_5}
Q_j^0 &= \frac{1}{2 \pi i} \int_{\Lambda_j} R^0(z) dz \quad \forall j \in \mathbb{N}, \quad 
Q_j = \frac{1}{2 \pi i} \int_{\Lambda_j} R(z) dz \quad \forall j > \ell \quad \text{and} \\
U_\ell   &=\frac{1}{2 \pi i} \int_{\Lambda_0} R(z) dz.
\end{align}
\begin{proposition} \label{0704_10}
Suppose the conditions (\ref{0704_1}) and  (\ref{0704_4}) hold with $\alpha \in (0, \infty) \backslash \{1 \}$. 
Then $\text{Sp} L$ is discrete and eventually simple.
Furthermore, with the notation
(\ref{0704_8})-(\ref{0705_5}), we have:
\begin{align*}
\text{Sp}L  &\subset \Pi_0 \cup \left( \cup_{j= \ell+1}^{\infty} \Pi_j \right) , \quad \\
\text{dim} U_\ell &= \sum_{j=1}^{\ell} \text{dim} Q_j^0, \quad \text{and} \quad \text{dim} Q_j = \text{dim} Q_j^0 = 1 \quad \forall j > \ell.
\end{align*}
\end{proposition}
\begin{proposition} \label{0722_4}
Fix $n \in \mathbb{N}$ with $n > \ell$. Then for each $z \in \Lambda_n$ we have:
\begin{align} \label{0722_5}
\|R(z) \| \leq 
\begin{cases} 
 \kappa n^{1-\alpha} \quad \text{if} \quad 1/2 < \alpha < 1, \\
 \kappa (n-1)^{1-\alpha} \quad \text{if} \quad 1 < \alpha < \infty.
\end{cases}
\end{align}
\end{proposition}
\begin{theorem} \label{0704_11}
Let $\alpha \in (1/2, \infty) \backslash \{1 \}$
and suppose the condition (\ref{0704_1}) and  (\ref{0704_4}) hold. 
Then there is a bounded operator $W$  such that
$W U_\ell W^{-1}   =  \sum_{k=1}^{\ell}  Q_k^0 $
 and $W Q_k W^{-1}  = Q_k^0$ for all $k >\ell$. 
 Hence, $\{U_\ell, Q_{\ell+1}, Q_{\ell+2}, \ldots \}$ is a Riesz system of projections.
\end{theorem}

Let us notice that Propositions \ref{0704_10}, \ref{0722_4} and Theorem \ref{0704_11} can be reformulated in 
an proper way for $\alpha = 1$. This would necessitate additional notation.  We refer the 
reader to our previous paper \cite{admt2} where the case $\alpha = 1$ 
is formulated and proven in detail.


\section{Technical preliminaries}
Define $B(\ell^2(\mathbb{N}))$ to be the space of all bounded linear operators on $\ell^2(\mathbb{N})$.
Given a strictly increasing sequence of real numbers $a = (a_k)$ define the generalized
discrete Hilbert transform (GDHT) by 
\begin{align} \label{2_0615}
(G_a \xi)(n) = \sum_{k \neq n} \frac{\xi_k}{a_k - a_n}, \quad \xi = (\xi_k)_{k=1}^{\infty}.
\end{align}
Of course care must be taken to ensure that the right hand side of  (\ref{2_0615}) is defined.
If $a_k = k \quad \forall k$ we have the standard discrete Hilbert transform (DHT) $G \in B(\ell^2(\mathbb{N}))$ (see HLP).

\begin{lemma} \label{a}
Suppose $a_{k+1}-a_k>0$ and $a_k \in \mathbb{N} \quad \forall k \in \mathbb{N}$. Then $G_a \in B(\ell^2(\mathbb{N}))$.
\end{lemma}
\begin{proof}
Suppose $\xi \in \ell^2 := \ell^2(\mathbb{N})$. Define the operator $I_a$ by $I_a(e_{a_k}) = e_{a_k} \quad \forall k$ 
and $I_a(e_j) = 0 \quad \forall j \notin a$ and define a vector $\tilde{\xi}$ by
$\tilde{\xi}_{a_k} = \xi_k \quad \forall k$ 
and $\tilde{\xi}_j = 0 \quad \forall j \notin a$. Then $\| \xi \| = \| \tilde{\xi} \|$ and 
$G_a \xi = I_a G I_a \tilde{\xi}$. Because $I_a$ and $G$ are bounded, $G_a$ is bounded as well.
\end{proof}

\begin{lemma} \label{b}
Suppose $A$ is an operator whose matrix entries satisfy 
\begin{align*}
|A_{k,j} | \leq \frac{C}{|k-j|^2}, \quad A_{k, k} = 0.
\end{align*}
Then $A \in B(\ell^2(\mathbb{N}))$ with $\| A \| \leq C \pi^2/3$.
\end{lemma}
\begin{proof}
The proof is elementary (see for example Lemma 4 in \cite{admt2}).
\end{proof}

\begin{lemma} \label{c}
Suppose 
\begin{align} \label{3_0615}
a_{k+1}- a_k > \delta \quad \forall k.
\end{align}
Then $G_a \in B(\ell^2(\mathbb{N}))$ with
$\| G_a \| \leq \frac{1}{\delta} \left( \frac{2 \pi^2}{3} + 2 \|G \| \right)$.
\end{lemma}
\begin{proof}
Write $\mathbb{R} = \cup_{k \in \mathbb{Z}} I_k, \quad I_k = [(k-1/2) \delta/2, (k+1/2) \delta/2)$.
Then by (\ref{3_0615}), $\# (I_k \cap a) = 0 \, \text{or} \, 1$.
Enumerate
 $\{  j \left( \frac{\delta}{2} \right) : \# (I_j \cap a) = 1 \}$ in increasing order and call the sequence $\tilde{a}$.
It follows from Lemma \ref{a} that the GDHT $G_{\tilde{a}} \in B(\ell^2(\mathbb{N}))$ with $\| G_{\tilde{a}} \| \leq (2/ \delta) \| G \|$.
Thus, with $A:=G_a - G_{\tilde{a}}$ it suffices to show $\|A \| \leq \frac{2 \pi^2}{3 \delta}$.

Consider the matrix entries $A_{k,k} = 0 \quad \forall k$, and for $j \neq k$,
\begin{align*}
|A_{j, k}| = \left|  \frac{1}{a_j - a_k} - \frac{1}{\tilde{a_j} - \tilde{a_k}}  \right| = 
\left|  \frac{(\tilde{a_j} - a_j) - (\tilde{a_k} - a_k)}{(a_j - a_k)(\tilde{a_j} - \tilde{a_k})}  \right| .
\end{align*}
By (\ref{3_0615}) we have $|a_j - a_k| > |j-k|\delta$, $|\tilde{a_j} - \tilde{a_k}| \geq |j-k| (\delta/2)$,
and $|\tilde{a_j} - a_j|, |\tilde{a_k} - a_k | < \delta/2$.
Hence,
\begin{align*}
|A_{j,k}| \leq  \frac{\delta/2 + \delta/2}{(|j-k| (\delta/2))(|j-k| \delta)} = (2/\delta) \frac{1}{|j-k|^2}.
\end{align*}
So by Lemma \ref{b}, $\|A \| \leq \frac{2 \pi^2}{3 \delta}$.
\end{proof}

\begin{lemma} \label{d}
Suppose $a_k$ is a strictly increasing sequence with $a_k \uparrow \infty$ and $(z_k)$ is a complex sequence satisfying 
\begin{align}
\notag 
&| \text{Im} z_k|  < \delta, 
\\
\label{4b_0615} 
&\text{Re} z_k \in (a_{k-1} + \delta, a_{k+1} -\delta)  \quad \text{and}
\\
\notag
&|\text{Re} z_k - a_k | < \Delta \quad \forall k.
\end{align}
Then the operator $Z_a$ defined by 
\begin{align*}
(Z_a \xi)(n) = \sum_{k \neq n} \frac{\xi_k}{a_k - z_n}
\end{align*}
is bounded in $\ell^2$ with 
\begin{align}
\|Z_a \| \leq \frac{1}{2 \delta} \| G \| + \frac{2 \delta \pi^2}{3 \Delta^2}.
\end{align}
\end{lemma}

\begin{proof}
 By (\ref{4b_0615}), $a_{k+1} - a_k > 2 \delta \quad \forall k$. Hence, by Lemma \ref{c}, 
 $\| G_a \| \leq \frac{1}{2\delta} \| G \|$. Now, set $A = G_a - Z_a$. It suffices to show 
 $\| A \| \leq \frac{ 2 \delta \pi^2}{3 \Delta^2}$. Consider the matrix elements $A_{k,k} = 0 \quad \forall k$ and 
 for $j \neq k$,
\begin{align*}
|A_{j,k}| &= \left| \frac{1}{a_k -a_n} - \frac{1}{a_k - z_n} \right| = \left| \frac{a_n - z_n}{(a_k - a_n)(a_k - z_n)}\right| \\
          & \leq \frac{2\delta}{\Delta |k-n| \Delta |k -n|} = \frac{2 \delta}{\Delta^2|k-n|^2}.
\end{align*}
It follows from Lemma \ref{b} that $\| A \| \leq \frac{2 \delta \pi^2}{3 \Delta^2}$.
\end{proof}

Define  $\ell^2(H)$ with the norm
\begin{align*}
\| \xi \|_{\ell^2(H)}^2 = \sum_{j=1}^{\infty} \| \xi_j \|^2_H, \quad \xi = (\xi_k), \xi_k \in H.
\end{align*}

\begin{lemma} \label{e}
Supppose $a$, $z$, and $Z_a$ are as in Lemma \ref{d}. Consider the operator $Z_a^V$ in $\ell^2(H)$
\begin{align*}
(Z_a^V \xi)(n) = \sum_{k \neq n} \frac{\xi_k}{a_k - z_n}.
\end{align*}
Then $\| Z_a^V \|_{\ell^2(H)} = \| Z_a \|_{\ell^2} \leq \frac{1}{2 \delta} \| G \| + \frac{2 \delta \pi^2}{3 \Delta^2}.$
\end{lemma}

\begin{proof}
Suppose $\xi  = (\xi_k) \in \ell^2(H)$ with $\xi_j = \sum_{k=1}^{\infty} \xi_j^{(k)} \phi_k \in H$.
\\
Then 
\begin{align*}
\| Z_a^V \xi \|^2 &= \sum_{n=1}^{\infty} \| (Z_a^V \xi)(n) \|^2 \\
                  &= \sum_{n=1}^{\infty} \sum_{k=1}^{\infty} \left|  \sum_{j \neq n} \frac{\xi_j^{(k)}}{a_j - z_n}\right|^2 
		   = \sum_{k=1}^{\infty} \sum_{n=1}^{\infty} | (Z_a \xi^{(k)})(n) |^2 \\
		  &=\sum_{k=1}^{\infty} \| Z_a \xi^{(k)} \|_{\ell^2}^2 \leq \|Z_a \|^2 \sum_{k=1}^{\infty} \| \xi^{(k)} \|_{H}^2 \\
		  &= \| Z_a \|^2 \| \xi \|_{\ell^2(H)}^2.
\end{align*}
\end{proof}
We now move to a series  of lemmas which will be used in the proofs of 
Proposition \ref{0704_10} and Theorem \ref{0704_11}. The proofs of these lemmas for values of $0< \alpha <1$, 
and $\alpha > 1$ follow a similar pattern so we only present proofs for values of $\alpha < 1$.
\begin{lemma} \label{0704_12}
Suppose $\alpha >0$, $\alpha \neq 1$, $\{ t_k \}_{1}^{\infty}$ satisfies 
(\ref{0704_1}),  $m \in V_M$ and $n-1 \in V_N$ with $M \leq N-2$. 
Then $t_n - t_m \geq c\left( 1- 1/v \right) v^{ \alpha N}$ with $v \in (\ref{0704_2})$.
\end{lemma}
\begin{proof}
We have 
\begin{align*}
t_n - t_m &= (t_n - t_{n-1}) + (t_{n-1} - t_{n-2}) + \ldots + ( t_{m+1} - t_m) \\
          &\geq c[ (n-1)^{\alpha-1} + (n-2)^{\alpha-1} + \ldots + m^{\alpha -1} ].
\end{align*}
Suppose first $\alpha \in (0,1)$. By the mean value theorem 
if $ a < b$ we have $\alpha a^{\alpha - 1} \geq b^\alpha - a^\alpha \geq \alpha b^{\alpha - 1}$.
\\
Hence, 
\begin{align*}
c[ (n-1)^{\alpha-1} &+ (n-2)^{\alpha-1} + \ldots + m^{\alpha -1} ] \geq (c/\alpha)[n^\alpha - m^\alpha] \\
	\notag	&\geq (c/\alpha)[v^{N\alpha} - v^{(N-1) \alpha}] \geq (c/\alpha) \alpha v^{N(\alpha -1)}[ v^N - v^{N-1}] \\
	\notag	&=c ( 1-1/v) v^{\alpha N}.
\end{align*}
A similar argument can be used for $\alpha >1$. We omit the details.
\end{proof}

The following lemma generalizes the boundedness of the discrete Hilbert transform.
It is a basic tool in our proof of Theorem \ref{0704_11}. In fact, 
our proof of Theorem \ref{0704_11} only works for values of $\alpha > 1/2$ because the following lemma does not 
hold for $\alpha \leq 1/2$ (see Remark \ref{0404_1}).  

\begin{lemma} \label{0704_13}
Suppose $\alpha \in (1/2,\infty)$ and  $\{ t_k \}_{1}^{\infty}$ satisfies (\ref{0704_1}).
Then there is a constant $\tilde{C}>0$ depending only on $\alpha$ such that  for any $(b_m) \in \ell^2(\mathbb{N})$ we have 
\begin{align} \label{0720_1}
\sum_{n=1}^{\infty} &\left| \sum_{m \neq n} \frac{m^{\alpha -1} b_m}{t_m - t_n}\right|^2 \leq \tilde{C}  \|b\|^2.
\end{align}
\end{lemma}
\begin{remark} \label{0404_1}
Of course, (\ref{0720_1}) does not hold if $\sum 1/t_n^{2} = \infty$
(even for $b = e_1$, i.e., $b(1) = 1,\quad  b(m) = 0, \, m>1$) so $t_n = n^a,\quad 0< a \leq 1/2$ or sequences with the growth condition (\ref{0704_1}) with $a \leq 1/2$ 
could not be analyzed with some analog of  Lemma \ref{0704_13} or 
Lemma \ref{0704_17}.
\end{remark}
\begin{proof}
Let $b \in \ell^2(\mathbb{N})$.
First suppose $1/2 < \alpha < 1$ so that $v$ from (\ref{0704_2}) 
can be written as $v = 2^{2(1+\delta)}$ with $\delta>0$.
Set $\gamma = \frac{1 + 2^{2 \delta}}{2}$.
By Cauchy's inequality we have 
\begin{align}
\label{0704_16}
\sum_{n=1}^{\infty} &\left| \sum_{m \neq n} \frac{m^{\alpha -1} b_m}{t_m - t_n}\right|^2
=\sum_{N=1}^{\infty} \sum_{n \in V_N} \left| \sum_{M=1}^{\infty} \sum_{m \in V_M} \frac{m^{\alpha -1} b_m}{t_m - t_n}\right|^2 \\
\notag
&\leq \left( \frac{2}{1-\gamma^{-1/2}}\right) \sum_{N=1}^{\infty} \sum_{n \in V_N} \sum_{M=1}^{\infty} \gamma^{|N-M|} \left| 
\sum_{m \in V_M} \frac{m^{\alpha -1} b_m}{t_m - t_n} \right|^2 
 \\
\notag
&=\left( \frac{2}{1-\gamma^{-1/2}}\right) (S_1 + S_2)
\end{align}
where
\begin{align*}
S_1 &= \sum_{N=1}^{\infty} \sum_{n \in V_N} \sum_{\substack{M=1 \\ |M-N|>1}}^{\infty} \gamma^{|N-M|} \left| 
\sum_{m \in V_M} \frac{m^{\alpha -1} b_m}{t_m - t_n} \right|^2 
\\
S_2 &=
\sum_{N=1}^{\infty} \sum_{n \in V_N} \sum_{\substack{M=1 \\ |M-N| \leq 1}}^{\infty} \gamma^{|N-M|} \left| 
\sum_{m \in V_M} \frac{m^{\alpha -1} b_m}{t_m - t_n} \right|^2.
\end{align*}

By Lemma \ref{0704_12} and another application of Cauchy's inequality
\begin{align}
\label{0704_14}
S_1 &\leq 
\left[
\sum_{N=1}^{\infty} \sum_{n \in V_N} \sum_{\substack{M=1 \\ |M-N|>1}}^{\infty} \gamma^{|N-M|}  
\sum_{m \in V_M} \frac{m^{2(\alpha -1)} }{(t_m - t_n)^2} \right]   \sum_{m \in V_M} b_m^2
\\
\notag
&\leq  \sum_{M=1}^{\infty} 
\sum_{m \in V_M} b_m^2 \sum_{\substack{N=1 \\ |N-M|>1}}^{\infty}  \frac{\gamma^{|N-M|} (\#V_M) \cdot (\#V_N) 
2^{-2M}}{(v/2)^{2 \text{max}(M,N)} }
\\
\notag
&\leq \sum_{M=1}^{\infty} 
\sum_{m \in V_M} b_m^2 \sum_{\substack{N=1 \\ |N-M|>1}}^{\infty}  \gamma^{|N-M|} \, v^{M+N+2} \, 2^{-2M} \, (2/v)^{2 \text{max}(M,N)} .
\end{align}
We will show that the following is uniformly bounded in $M$
\begin{align*}
\sum_{N=1}^{\infty} \gamma^{|N-M|} \, v^{-|M-N|} \, 2^{2 \text{max}(M,N) - 2M}.
\end{align*}
We have
\begin{align*}
\sum_{N=1}^{\infty} &\gamma^{|N-M|} \, v^{-|M-N|} \, 2^{2 \text{max}(M,N) - 2M} \\
&= \sum_{N=1}^{\infty} \gamma^{|N-M|} \, 2^{-2(1+\delta)|M-N| - 2M} \, 2^{2 \text{max}(M,N)} \\
&\leq \sum_{N=1}^{\infty} \gamma^{|N-M|} \, 2^{-2 \text{max}(M,N) - 2 \delta|M-N|} \, 2^{2 \text{max}(M,N)} \\ 
&= \sum_{N=1}^{\infty} \gamma^{|N-M|} \, \left(2^{-2 \delta}\right)^{|M-N|} \leq \frac{2}{1-\frac{\gamma}{2^{2 \delta}}}.
\end{align*}
Combining this bound with (\ref{0704_14}), we conclude
\begin{align} \label{0704_15}
S_1 \leq \left( \frac{2}{1-\frac{\gamma}{2^{2 \delta}}} \right) \|b \|^2.
\end{align}

Now
\begin{align*}
S_2 &\leq \sum_{N=1}^{\infty} \sum_{n \in V_N} \sum_{ |M-N| \leq 1} \gamma^{|M-N|} 
\left( \sum_{m \in V_M} \frac{m^{2(\alpha -1)}}{(t_m - t_n)^{2}}\right) \left(\sum_{m \in V_M} |b_m|^2 \right) \\
&\leq 
\sum_{N=1}^{\infty} \sum_{ |M-N| \leq 1} (\# F_M) \cdot  (\# F_N) \, \gamma^{|M-N|} \,
2^{-2M} \left( v/2 \right)^{-2 \text{max}(M,N)}\left(\sum_{m \in V_M} |b_m|^2 \right) \\
&\leq 
\sum_{N=1}^{\infty} \sum_{ |M-N| \leq 1} v^{M+N+2} \, \gamma^{|M-N|} \,
 v^{-2 \text{max}(M,N)} \, 2^{-2M+2 \text{max}(M,N)} \left(\sum_{m \in V_M} |b_m|^2 \right) \\
&\leq 
16 \, v^4 \sum_{N=1}^{\infty} \sum_{|M-N|\leq 1} \left( \sum_{m \in V_M} |b_m|^2 \right) 
\leq 3\cdot 16  \cdot v^4 \gamma \| b \|^2.
\end{align*}

Combining these bounds with (\ref{0704_16}) and (\ref{0704_15})
we have
\begin{align*}
\sum_{n=1}^{\infty} &\left| \sum_{m \neq n} \frac{m^{\alpha -1} b_m}{t_m - t_n}\right|^2
\leq 
\left( \frac{2}{1-\gamma^{-1/2}}\right) \left(\left( \frac{2}{1-\frac{\gamma}{2^{2 \delta}}} \right) + 3\cdot 16  v^4 \gamma \right) \|b\|^2.
\end{align*}
So (\ref{0720_1}) is proved for $1/2 < \alpha < 1$.

The proof for $\alpha > 1$ is similar. We omit the details.
\end{proof}

The following lemma can be proven in the same manner as Lemma \ref{0704_13}.
We omit the details.
\begin{lemma} \label{0704_17}
Suppose $\alpha \in (1/2,\infty) \backslash \{1\}$ and  $\{ t_k \}_{1}^{\infty}$ satisfies (\ref{0704_1}).
Let $\{z_k \}$ be a sequence such that $|z_k -t_k | \leq (c/2) k^{\alpha-1} \forall k \in \mathbb{N}$. 
Then there is a constant $C>0$ depending only on $\alpha$ such that  for any $(b_m) \in \ell^2(\mathbb{N})$ we have 
\begin{align*}
\sum_{n=1}^{\infty} &\left| \sum_{m \neq n} \frac{m^{\alpha -1} b_m}{t_m - z_n}\right|^2 \leq C  \|b\|^2.
\end{align*}
\end{lemma}

The following vector-valued version of Lemma \ref{0704_17} can be proven in the same manner as Lemma \ref{e}.
We omit the details.
\begin{lemma} \label{0722_9}
Suppose $\alpha \in (1/2,\infty) \backslash \{1\}$ and  $\{ t_k \}_{1}^{\infty}$ satisfies (\ref{0704_1}).
Let $\{z_k \}$ be a sequence such that $|z_k -t_k | \leq (c/2) k^{\alpha-1} \forall k \in \mathbb{N}$. 
Then there is a constant $C>0$ depending only on $\alpha$ such that  for any $(b_m) \in \ell^2(H)$ we have 
\begin{align*}
\sum_{n=1}^{\infty} &\left| \sum_{m \neq n} \frac{m^{\alpha -1} b_m}{t_m - z_n}\right|^2 \leq C  \|b\|^2.
\end{align*}
\end{lemma}

\section{Proof of Proposition \ref{A} and Proposition \ref{0720_A}}
\begin{proof}
Let $z \notin  \cup_{0}^{\infty} \Pi_k$. To show $z \notin \text{Sp} L$ it suffices to show $\| B R^0(z) \| \leq 1/2$ since then
$R(z) = R^0(z) (I - BR^0(z))^{-1}$ is well defined.
To this end let $f \in H$ with $\| f \|^2 = 1, \quad f = \sum f_k \phi_k$. 
\\
We have
\begin{align*}
\| B R^0(z) f \|^2 &= \| B R^0(z) \sum f_k \phi_k \|^2 = \| \sum \frac{f B \phi_k}{z-t_k} \|^2 \\
		&\leq \| f \|^2  \sum \frac{\beta_k}{|z - t_k|^2} = \sum \frac{\beta_k}{|z-t_k|^2}.
\end{align*}

Consider first the case $\text{Re} z  \in F_m$, for some  $m \geq M+N+1$. 
We have 
\begin{align} \label{4}
\sum \frac{\beta_k}{|z-t_k|^2} &= \sum \frac{\beta_k}{(\text{Re} z - t_k)^2 + \text{Im}z^2} \\
\notag		&\leq \left( \frac{2p}{d}\right)^2 \left(\sum_{j \in F_m} \beta_j 
		+ \left[ \sum_{J=1}^{N} + \sum_{J=N+1}^{\infty} \right]   \frac{1}{J^2}\sum_{j \in F_{m\pm J}} \beta_j  
		\right).
\end{align}
Because $ m \geq M+N+1$, $m \pm N \geq M$ whenever $J < N$. So, by (\ref{2}) we have 
\begin{align} \label{5}
 \sum_{j \in F_m} \beta_j 
	+ \sum_{J=1}^{N}  \frac{1}{J^2}\sum_{j \in F_{m\pm J}} \beta_j \leq 
	p\left( 1 + 2 \sum_{J=1}^{N} \frac{1}{J^2} \right) \left(\frac{1}{8p(1+ \pi^2/3)} \right) \left( \frac{d}{2p}\right)^2
\end{align}
and by (\ref{3}) we have
\begin{align} \label{6}
\sum_{J=N+1}^{\infty}  \frac{1}{J^2}\sum_{j \in F_{m \pm J}} \beta_j \leq 
2p \left( \sum_{J=N+1}^\infty \frac{1}{J^2} \right) \left( \frac{d^2}{16 p^2} \right) \left( p \sum_{J=N+1}^{\infty} \frac{1}{J^2}
\right).
\end{align}
Combining (\ref{4}) with (\ref{5}-\ref{6}) we conclude that 
\begin{align*}
\| B R^0(z) f \|^2 \leq 1/4
\end{align*}
whenever $\text{Re} z \in F_m$ for some $m \geq M+N+1$.

Now consider the case $\text{Re} z < T_{M+N+1} + \frac{d}{2p}$,  i.e. $\text{Re} z \notin F_m \quad \forall m \geq M+N+1$.
Then $z \notin \cup_{0}^{\infty} \Pi_j$ implies $|z-t_k|^2 \geq h^2 + \text{Re} (|z| - t_k)^2.$
\\
Thus
\begin{align*}
\sum_{k=1}^{\infty} \frac{\beta_k}{|z-t_k|^2} &\leq \sum_{k=1}^{\infty} \frac{\beta_k}{h^2 + (|\text{Re}z| - t_k)^2} \\
 &=\sum_{J=1}^{\infty} \sum_{j \in F_J} \frac{\beta_k}{h^2 + (|\text{Re}z| - t_k)^2} \\
 &\leq 2 \sum_{J=0}^{\infty} 
 \frac{p \| \beta \|_{\infty}}{h^2 + (\frac{Jd}{2p})^2 } \leq  \frac{2p \| \beta \|_{\infty}}{8p \|\beta \|_{\infty}}
=1/4.
\end{align*}
So we have shown that $\|R^0(z) B \|^2 \leq 1/4$ for all $z \notin \Pi_0 \cup \cup_{k=K+1}^{\infty} \Pi_k.$ 
Thus, $\text{Sp} L \subset \Pi_0 \cup \cup_{k=K+1}^{\infty} \Pi_k.$
Also,  
$\| R(z) \|  = \| R^0(z)(I - BR^0(z))^{-1} \| \leq \|R^0(z) \| (1/2) \leq d/p$ for all $z \notin \Pi_0 \cup \cup_{k=K+1}^{\infty} \Pi_k$.

A standard argument (see \cite{naimark}) shows that 
\begin{align*}
\text{Trace} \frac{1}{2 \pi i} \int_{ \Gamma_n} \left( z - T - tB \right)^{-1} dz, \quad 0 \leq t \leq 1,
\end{align*}
is a continuous integer-valued scalar function so it is constant and (\ref{0722_1} - \ref{0722_2}) hold. 
\end{proof}

\section{Proof of Theorem \ref{B}}

We first reproduce Lemma 4.17(a) from \cite{katobook}. See also \cite{katopaper}.
\begin{lemma} \label{0710_1}
Let $\{Q^0_k\}_{j \in \mathbb{Z}_+}$ be a complete family of orthogonal projections in a 
Hilbert space $X$ and let 
$\{Q_k\}_{j \in \mathbb{Z}_+}$ be a family of (not necessarily orthogonal) projections such that
$Q_j Q_k = \delta_{j,k}Q_j$. Assume that 
\begin{align*}
\text{dim}(Q^0_0) = \text{dim}(Q_0) = m< \infty \\
\sum_{j=1}^\infty \| Q^0_j(Q_j-Q^0_j)u\|^2 \leq c_0 \|u\|^2, \quad \text{for every } \quad u \in X
\end{align*}
where $c_0$ is a constant smaller than $1$. Then there is a bounded operator $W: X \rightarrow X$ with bounded inverse
such that $Q_j = W^{-1} Q^0_j W$ for $j \in \mathbb{Z}_+$.
\end{lemma}

We are now ready to prove Theorem \ref{B}.
\begin{proof}
By Lemma \ref{0710_1}  it suffices to show $\exists N_* \in \mathbb{N}$ such that 
for all $f \in H$ with $\|f \| =1$,
\begin{align*}
\sum_{ n \geq N_*} \| P_n^0 ( P_n - P_n^0) f \|^2 \leq 1/2.
\end{align*}
Fix $n > K$ (with $K$ from Proposition \ref{A}) and $f = \sum f_k \phi_k \in H$ with $\| f \| =1$. Then
\begin{align*}
P_n - P^0_n = \frac{1}{2 \pi i} \int_{\Gamma_n} (R(z) - R^0(z) ) dz  = \frac{1}{2 \pi i} \int_{\Gamma_n} R(z) B R^0(z) dz.
\end{align*}
So by Proposition \ref{0720_A}, inequality  (\ref{0720_2}),
\begin{align} \label{0720_3}
\| ( P_n - P_n^0) f \|^2 &= \frac{1}{2 \pi} \| \int_{\Gamma_n} R(z) B R^0(z) f dz \|^2 \\
\notag	                 &\leq  \frac{1}{2 \pi} \left[ \int_{\Gamma_n} \| R(z) B R^0(z) f \| \right]^2
			  \leq  C \left[\int_{\Gamma_n} \| B R^0(z) f \| \right]^2 \\
\notag			 & = C \left[ \int_{\Gamma_n} 
			 \| 
			 \sum_{k=1}^{\infty} 
			 \frac{f_k B \phi_k}{z - t_k}
			 \| dz
			 \right]^2
\end{align}
where $C = \frac{1}{2 \pi}(d/p)^2$.

For $n \geq K$ define $z^*_n \in  \Gamma_n$ to be a point where the maximum of the sum 
\begin{align*}
\| \sum_{k=1}^{\infty} \frac{f_k B \phi_k}{ z - t_k} \|
\end{align*}
is attained.
Note that $(z^*_n)$ depends on $f$.  Since $|\Gamma_k| \leq 3C \quad \forall k \geq 1$, we have by (\ref{0720_3})
\begin{align} \label{0619_4}
\| (P_n - P_n^0) f \|^2 &\leq C | \Gamma_n |^2  \| \sum_{k=1}^{\infty} \frac{f_k B \phi_k}{ z - t_k} \|^2 \\
\notag
&\leq  C (3d)^2 \| \sum_{k=1}^{\infty} \frac{f_k B \phi_k}{ z - t_k} \|^2.
\end{align}

Suppose, for now, that $p=1$ so that $\# (\tau \cap F_k) \leq 1$ $\forall k \geq 1$. We will show 
that given any $\epsilon > 0$ if we choose $N_1$ as in (\ref{0619B}) below, 
then 
\begin{align} \label{0619_1}
 \sum_{ n \geq N_1} \| \sum_{k=1}^{\infty} \frac{f_k B \phi_k}{z^*_n - t_k} \|^2 < \epsilon \quad \forall \|f \| =1.
\end{align}
Note that $z^*_n \in  \Gamma_n$ depends on $f$.

Recall that $G$ is the cannonical discrete Hilbert transform and set 
\begin{align*}
C_1 = 4 \left( \frac{d \|G \|}{p} + \frac{ \pi^2}{3dp} \right).
\end{align*}
Select $M_1$ large enough so that
\begin{align}
\label{0619A}
\| B \phi_k \|^2 \leq \epsilon/C_1 \quad \forall k \geq M_1
\end{align}
and $N_1$ large enough so that  whenever $w \in  \Gamma_n \quad \forall n \geq K$
\begin{align}
\label{0619B}
\sum_{n = N_1}^{\infty} |w - t_{M_1}|^{-2} \leq  \frac{\epsilon}{4 \| \beta \|_{\infty} M_1} \quad \forall m \leq M_1.
\end{align}
Then 
\begin{align*}
 \| \sum_{k=1}^{\infty} \frac{f_k B \phi_k}{z^*_n - t_k} \|^2 \leq 
 2 \left[ \| \sum_{k \leq M_1}\frac{f_k B \phi_k}{z^*_n - t_k} \|^2 + \| \sum_{k > M_1} \frac{f_k B \phi_k}{z^*_n - t_k} \|^2 \right].
\end{align*}
By Cauchy's inequality we have
\begin{align*}
\| \sum_{k \leq M_1}\frac{f_k B \phi_k}{z^*_n - t_k} \|^2 \leq  
\left( \sum_{k \leq M_1} |f_k|^2  \| B \phi_k \|^2 \right)\left( \sum_{k \leq M_1} |z^*_n - t_k|^{-2} \right).
\end{align*}
So by (\ref{0619B})
\begin{align} \label{0720_4}
2 \sum_{n \geq N_1} \| \sum_{k \leq M_1}\frac{f_k B \phi_k}{z^*_n - t_k} \|^2 \leq 
\| \beta \|_{\infty} M_1 \sum_{ n \geq N_1} |z^*_n - t_{M_1} |^{-2} < \epsilon/2.
\end{align}

It follows from Lemma \ref{e} and (\ref{0619A}) that
\begin{align} \label{0720_5}
2 \sum_{n \geq N_1} \| \sum_{ k > M_1} \frac{f_k B \phi_k}{z^*_n - t_k} \|^2 &\leq 
2 \| (f_k B \phi_k )_{k=M_1}^{\infty} \|_{\ell^2(H)}^2 C_1 \\
\notag
&\leq \sup_{k \geq M_1} \| B \phi_k \|^2 C_1 < \epsilon/2.
\end{align}
Hence, combining (\ref{0720_4}) and (\ref{0720_5}) we have proven (\ref{0619_1}).

Now suppose $p > 1$. Reindex the sequences $t_k$, $f_k$, and $\phi_k$ in such a way that 
for all $k \geq 1$
\begin{align*} 
\tau \cap F_k = \{ t_k^{(1)} \leq t_k^{(2)} \leq \ldots \leq t_k^{(J_k)} \}, \quad J_k \leq p.
\end{align*}
Then for $\tilde{K} \geq K$,
\begin{align*}
C\sum_{n \geq \tilde{K} } & \| \sum_{k=1}^{\infty} \frac{ f_k B \phi_k}{z^*_n - t_k} \|^2  = 
C\sum_{n \geq \tilde{K}}  \| \sum_{k=1}^{\infty} \sum_{j=1}^{J_k} \frac{ f^{(j)}_k B \phi^{(j)}_k}{z^*_n - t^{(j)}_k} \|^2 \\
&\leq  2^p C\sum_{j=1}^{p} \sum_{n \geq \tilde{K}}  \| \sum_{k=1}^{\infty}  \frac{ f^{(j)}_k B \phi^{(j)}_k}{z^*_n - t^{(j)}_k} \|^2.
\end{align*}
Note that if $J_k <p$ some terms in the series are taken to be $0$.
For each $j \leq p$ the sequence $t^{(j)}_1 \leq t^{(j)}_2 \leq \ldots $ satisfies
(\ref{1_0615}) with $p=1$. So by taking $\epsilon = 1/(2 \cdot 2^p C)$ and applying   (\ref{0619_1}) for each $j \leq p$
the Theorem is proven  by (\ref{0619_4}).
\end{proof}

\section{Proof of Proposition \ref{0704_10} and Proposition \ref{0722_4}}
Suppose that $z \notin \Pi_0 \cup \cup_{j= \ell+1}^{\infty} \Pi_j.$
We will show that $\|R^0(z) B\|^2 \leq 1/2$. It follows that 
\begin{align}
\label{0722_6}
R(z) = (I - R^0(z)B)^{-1} R^0(z)
\end{align}
is 
well defined. Let $f \in H$ with $\|f \| =1$, $f = \sum f_k \phi_k$. 
\\
Then 
\begin{align*}
\|R^0(z) B f \|^2 &= \| \sum f_j B R^0(z) \phi_j \|^2 \\
                  &=\| \sum f_j \frac{B \phi_j}{z-t_j} \|^2 \leq \left[ \sum |f_j| \frac{\| B \phi_j \|}{|z-t_j|} \right]^2 \\
		  &\leq \sum \frac{\|B \phi_j \|^2}{|z - t_j|^2} = \sum \frac{c_j^2 j^{2(\alpha - 1)}}{|z - t_j|^2}.
\end{align*}
Suppose first that $ \text{Re } z > t_\ell + (c_\ell / 2 )\ell^{\alpha -1}$ so that $\text{Re } z \in  [v^{\tilde{N}}, v^{\tilde{N}+1})$ 
for some $\tilde{N} > N$.
Then 
\begin{align*} 
\sum_{j=1}^{\infty} \frac{c_j^2 j^{2(\alpha -1)}}{|z - t_j|^2} = S_1 + S_2
\end{align*}
with
\begin{align*}
S_1 = \sum_{J=1}^{\tilde{N}-1-N/2} \sum_{j \in V_J} \frac{c_j^2 j^{2 (\alpha - 1)}}{|z - t_j|^2}, \quad  
S_2 = \sum_{J=\tilde{N} - N/2}^{\infty} \sum_{j \in V_J} \frac{c_j^2 j^{2 (\alpha - 1)}}{|z - t_j|^2}.
\end{align*}
We have 
\begin{align*}
S_1 \leq c_\infty^2 \sum_{J=1}^{\tilde{N}-1 -N/2 } \sum_{j \in V_J} \frac{j^{2(\alpha -1)}}{|z-t_j|^2}.
\end{align*}
If $0 < \alpha < 1$, then for each $J \leq \tilde{N}-1 - N/2 ,$
\begin{align*}
\sum_{j \in V_J} \frac{j^{2(\alpha -1)}}{|z - t_j|^2} &\leq \#F_J \frac{2^{-2 J}}{v^{2(\tilde{N}-1)\alpha}}  \\
&= \#F_J 2^{-2 J} (v/2)^{2} (v/2)^{-2\tilde{N}} \leq v^{J+1} 2^{-2J} (v/2)^{2} (2/v)^{2\tilde{N}} \\
&=(2/v)^{2(\tilde{N}-J) - 2} v^{-J} v \leq v (2/v)^{N} v^{-J}.
\end{align*}
It follows from (\ref{0704_6}) that
\begin{align*}
S_1 \leq c_\infty^2 v (2/v)^{N} \sum_{J=1}^{\tilde{N}-1-N/2} v^{-J} \leq c_\infty^2 (2/v)^N \left( \frac{1}{1-1/v} \right)  < 1/4.
\end{align*}

If $\alpha > 1$, then for each $J \leq  \tilde{N} -1 - N/2 $ we have 
\begin{align*}
\sum_{j \in V_J} \frac{j^{2(\alpha - 1)}}{|z - t_j|^2} &\leq \#V_J \frac{2^{2J}}{v^{2(\tilde{N}-1) \alpha}} 
\leq v^{J+1} \frac{2^{2J}}{(2v)^{2(\tilde{N}-1)}} \\
&= v^{J- 2\tilde{N} + 3} 2^{2J - 2\tilde{N} + 2} \leq v (2v)^{2J - 2\tilde{N} + 2}v^{-J} \\
&\leq v ( 2v)^{2(J-\tilde{N})+2} v^{-J} \leq v (2v)^{-N} v^{-J}.
\end{align*}
So it follows from (\ref{0704_6}) that 
\begin{align*} 
S_1 \leq c_\infty^2 v (2v)^{-N} \sum_{J=1}^{\tilde{N}-1-N/2 } v^{-J}  \leq c_\infty^{2} (2v)^{-N} \left( \frac{1}{1-1/v} \right) < 1/4.
\end{align*}
Now
\begin{align*}
S_2 &= \sum_{J = \tilde{N}- N/2}^{\infty} \sum_{j \in V_J} \frac{c_j^2 j^{2 (\alpha -  1)}}{|z - t_j|^2} \\
&\leq \sup_{\substack{j \in V_J \\ J \geq \tilde{N} - N/2}} 
\left[
\left( \sum_{\substack{J=\tilde{N}-N/2 \\ J \neq 
\tilde{N}-1, \tilde{N}, \tilde{N}+1}}^{\infty} + \sum_{J=\tilde{N}-1, \tilde{N}, \tilde{N}+1} \right) \sum_{j \in V_J} 
\frac{j^{2(\alpha -1 )}}{|z - t_j|^2}
\right].
\end{align*}
Let $\text{Re}z \in (t_k, t_{k+1}], k \in V_{\tilde{N}-1} \cup V_{\tilde{N}} \cup V_{\tilde{N}+1}$. Then 
\begin{align*}
|z - t_k|^2 \geq (v/2)^2 k^{2(\alpha -1)} \geq (v/2)^2 2^{2(\tilde{N}-1)}
\end{align*}
and for $j \neq k$, $j \in  V_{\tilde{N}-1} \cup V_{\tilde{N}} \cup V_{\tilde{N}+1}$ we have
\begin{align*}
|z - t_j|^2 \geq (\kappa/2)^2 |j - k|^2 2^{2(\tilde{N}-1)}.
\end{align*}
Thus, 
\begin{align*}
\sum_{J= \tilde{N}-1,\tilde{N},\tilde{N}+1} \sum_{j \in V_J} &\frac{j^{2(\alpha -1)}}{|z - t_j|^2}  \\
&\leq 
\frac{2^{2(\tilde{N}+1)}}{(\kappa/2)^2 2^{2(\tilde{N}-1)}}+\sum_{J= \tilde{N}-1,\tilde{N},\tilde{N}+1} 
\sum_{\substack{j \in V_J\\ j \neq k}} \frac{2^{2(\tilde{N}+1)}}{(v/2)^2|j-k|^2 2^{2(\tilde{N}-1)}}
\\
&\leq 
\frac{16}{c^2} \left(1+ 2 \sum_{j=1}^{\infty} j^{-2} \right) = \frac{16}{\kappa^2} \left( 1 + 2 \pi^2/3 \right).
\end{align*}
Furthermore, for $0< \alpha < 1$
\begin{align*}
\sum_{J= \tilde{N} - N/2}^{\infty} \sum_{j \in V_J} \frac{j^{2(\alpha - 1)}}{|z - t_j|^2} 
&\leq \sum_{\substack{J = \tilde{N} - N/2 \\ J \neq \tilde{N}-1, \tilde{N}, \tilde{N}+1}}^{\infty} 
\frac{(\#F_J) \,  2^{-2J}}{(v/2)^{2(J-1)}} \\
&\leq 4 \sum_{J=\tilde{N}-M}^{\infty} \frac{v^{J+1}}{v^{2J-2}} \leq 4v \sum_{J=1}^{\infty} v^{-J} \leq 4\left( \frac{1}{1-1/v}\right).
\end{align*}

By a similar argument, if $\alpha > 1$ we also have
\begin{align*}
\sum_{J= \tilde{N} - N/2}^{\infty} \sum_{j \in V_J} \frac{j^{2(\alpha - 1)}}{|z - t_j|^2} \leq 4\left( \frac{1}{1-1/v}\right).
\end{align*}
Hence, by (\ref{0704_7})
\begin{align*}
S_2 \leq  \sup_{\substack{j \in V_J \\ J \geq \tilde{N} - N/2}}  \left((16/\kappa^2)(1+ 2 \pi^2/3) +\frac{4}{1-1/v}  \right) < 1/4.
\end{align*}
Now suppose that $\text{Re} z \leq t_\ell + ( \kappa/2 ) \ell^{\alpha -1}$ so that 
$\text{dist}(z,\tau)^2 \geq Y^2$.
Suppose $0 < \alpha < 1.$ Then 
\begin{align*}
\sum_{j=1}^{\infty} &\frac{c_j^2 j^{2(\alpha -1)}}{|z - t_j|^2}  
\leq c_\infty^2 \sum_{J=1}^{N} \frac{ (\#V_J) 2^{-2J}}{Y^2} + 
\sup_{\substack{j \in V_J \\ j \geq N}} c_j^2  \sum_{J=N+1}^{\infty} \frac{(\#V_J) 2^{-2J}}{v^{2\alpha(J-1)}} \\
&\leq \frac{c_\infty^2 v}{Y^2} \sum_{J=1}^{N} (\sqrt{v}/2)^{2J} + 
\sup_{\substack{j \in V_J \\ J \geq N}} \sum_{J=N+1}^{\infty} (2/v)^2 v^{1-J} < 1/2
\end{align*}
by (\ref{0704_7}) and (\ref{0704_8}).
We have shown 
\begin{align} \label{0722_7}
\| (I - R^0(z)B)^{-1} \| \leq 2
\end{align}
so that by (\ref{0722_6}) $R(z)$ is well-defined.

By a similar argument for $\alpha > 1$ we also have 
\begin{align*}
\sum_{j=1}^{\infty} &\frac{c_j^2 j^{2(\alpha -1)}}{|z - t_j|^2}  
\leq 1/2.
\end{align*}
We omit the details.

Now, definition (\ref{0704_5}) implies that for $z \in \Lambda_n$,
\begin{align} \label{0722_8}
\|R^0(z) \| \leq 
\begin{cases} 
(\kappa/2) n^{1-\alpha} \quad \text{if} \quad 1/2 < \alpha \leq 1, \\
(\kappa/2) (n-1)^{1-\alpha} \quad \text{if} \quad 1 < \alpha < \infty.
\end{cases}
\end{align}
Hence, inequality (\ref{0722_5}) follows from (\ref{0722_6}) and (\ref{0722_7}) together with 
(\ref{0722_3}) and (\ref{0722_8}).

\section{Proof of Theorem \ref{0704_11}}
\begin{proof}
For the case of $\alpha = 1$, this theorem is proven in the paper \cite{admt2}.  
Henceforth, we assume $\alpha \neq 1$.
By Lemma \ref{0710_1} it suffices to show there exists an integer $N_*$ such that 
\begin{align} \label{0722_12}
\sum_{n \geq N_*} \| Q_n^0 (Q_n - Q_n^0) f \|^2 \leq 1/2.
\end{align}
Fix $n > N$ and $f = \sum f_k \phi_k \in H$ with $\| f \| = 1$.
Then 
\begin{align*}
Q_n - Q^0_n &= \frac{1}{2 \pi i} \int_{\Lambda_n} (R(z) - R^0(z)) dz \\
	    & = \frac{1}{2 \pi i} \int_{\Lambda_n} R(z) B R^0(z) dz.
\end{align*}
Hence, 
\begin{align*}
\| (Q_n - Q_n^0) \|^2 &= \frac{1}{2 \pi } \| \int_{ \Lambda_n} R(z) B R^0(z)f dz \|^2 \\
		& \leq \frac{1}{2\pi} \left[ \int_{\Lambda_n} \| R(z) B R^0(z) f \| dz \right]^2 \\
		&  =   \frac{1}{2\pi} \left[ \int_{\Lambda_n} \| \sum_{k=1}^{\infty} \frac{f_k R(z) B \phi_k}{z-t_k}\|  dz\right]^2. 
\end{align*}
Now define $z^*_n \in \Lambda_n$ to be a point at which the following sum attains its maximum,
\begin{align*}
\| \sum_{k=1}^{\infty} \frac{f_k B \phi_k}{z - t_k} \| \quad z \in \Lambda_n.
\end{align*}
Combining (\ref{0722_3}) with (\ref{0722_5}) yields
\begin{align*}
| \Lambda_n |^2 \|R(z) \|^2 \leq 16 \kappa^2. 
\end{align*}
So,
\begin{align*}
\| (Q_n - Q_n^0) \|^2 \leq 
\frac{16 \kappa^2}{2\pi} \left[\| \sum_{k=1}^{\infty} \frac{f_k B \phi_k}{z^*_n-t_k}\|  \right]^2.
\end{align*}
Recall the constant $C$ from Lemma \ref{0722_9}. 
Condition (\ref{0704_4}) implies that there exists an absolute constant $N_*$ such that 
 \begin{align} \label{0722_11}
 \|B \phi_k \| \leq \frac{2 \pi}{32 \kappa^2 C} k^{\alpha - 1} \quad \forall \, k \geq N_*.
 \end{align}
 Thus
\begin{align} \label{0722_10}
\sum_{n \geq N_*} \| Q_n^0 (Q_n - Q_n^0) f \|^2 &\leq 
\sum_{n \geq N_*} \| (Q_n - Q_n^0) f \|^2  \\
\notag              
&\leq \frac{16 \kappa^2}{2\pi} \sum_{n \geq N_*} \left[\| \sum_{k=1}^{\infty} \frac{f_k B \phi_k}{z^*_n-t_k}\|  \right]^2. 
\end{align}
Finally, combining (\ref{0722_10}) with Lemma \ref{0722_9} and (\ref{0722_11}) yields (\ref{0722_12}) and the proof is complete.
\end{proof}
\section{Further remarks}
\subsection{} \label{0404_2}
Our statement of Theorem \ref{0704_11} required  the condition  
\begin{align} \label{0404_3}
\lim_{k \rightarrow \infty} c_k =0
\end{align}
where $\{c_k\}$ is defined in (\ref{0704_4}).
With a careful accounting of 
quantities appearing in the proof of Theorem \ref{0704_11} we could have written a constant $c^*$ such that 
the condition (\ref{0404_3}) could be
replaced by the weaker condition  
\begin{align} \label{0404_4}
\lim \sup c_k \leq c^*.
\end{align}

However, condition  (\ref{0404_3}) or (\ref{0404_4}) could not be weakened in a significant way: an 
assumption $\lim \sup c_k < \infty$ would not guarantee the statement of Theorem \ref{0704_11}. 
A counterexample in the case $\alpha = 1$ is given in \cite{admt2}, Section 6.3.

Now we'll adjust the constructions of \cite{admt2} to get an operator $B$, 
with 
\begin{align} \label{0405_1}
\sup_{m} \{ \|B \phi_k \| (t_{2m} - t_{2m-1})^{-1},\quad  k=2m-1,\, 2m \} = 1/2
\end{align}
such that the perturbation $L=T+B$ has a discrete spectrum, all points of $\text{Sp}(T+B)$ are simple 
eigenvalues, the system $\{ \psi_k \}$ of eigenvectors of $L$ is complete, but it is \emph{not} a basis in $H$.
If $t_n = n^\alpha, \quad 0 < \alpha <\infty$, then (\ref{0405_1}) guarantees
that $c_\infty  \leq 1/2$.

Special $2$-dimensional blocks play an important role in this construction. 
\\
Put 
\begin{align} \label{0405_2}
b = \begin{bmatrix}
\phantom{-}0 & s \\
-s & 0
\end{bmatrix}, \quad 0 < s < 1, \quad 
s^2 + h^2 = 1, \quad 0 < h << 1.
\end{align}
This choice is a slight adjustment of a $2$-dimensional 
block (64) in \cite{admt2}. It simplifies elementary calculations of the 
$\text{Angle}(g^+,g^-)$, etc., for example. 
Such a block (\ref{0405_2}) could be used to get the same counterexample in 
Section 6.3, \cite{admt2} instead of (64) there. 
Of course, $\| b \| = s$ in $\mathbb{C}^2$ in the Euclidean norm. 

We have
\begin{align*}
\begin{bmatrix}
0 & 0 \\
0 & 2
\end{bmatrix}
+ b = \begin{bmatrix} 1 & \phantom{-}0 \\ 0 & -1 \end{bmatrix} + c, 
\quad c = \begin{bmatrix}-1 & s \\ -s & 1 \end{bmatrix}
\end{align*}
and
\begin{align*}
cg^{\pm} &= \pm h g^{\pm} \quad \text{where} \\
g^{\pm} &= (1,  G^{\pm 1}), \quad G = \sqrt{ \frac{1+h}{1-h} }.
\end{align*}
If $\alpha = \text{Angle}(g^+, g^-)$ then
\begin{align*}
(\cos \alpha )^2 = \frac{(g^+, g^-)^2}{ \| g^+ \|^2 \cdot \|g^- \|^2} = 1-h^2 = s^2.
\end{align*}
So $\sin \alpha = h$.

If $f = \Phi_0(f) u_0 + \Phi_1(f) u_1$ is the standard basis decomposition in $\mathbb{C}^2$ then 
\begin{align} \label{0404_5}
\| \Phi_0 \| = \| \Phi_1 \| = 1/ \sin \alpha = 1/h.
\end{align}
Now we define $B = \{  b(m) \}$ where $b(m)$ are $2$-dimensional blocks 
\begin{align*}
\frac{1}{2} (t_{2m} - t_{2m-1} ) \begin{bmatrix} \phantom{-}0 & s \\ -s & 0\end{bmatrix}, \quad  s= s(m), 
\quad \text{say} \quad &s(m)^2 + (1/m)^2 = 1, 
\end{align*}
on $\mathbb{C}^2 = E_m := \text{Span} \{ \phi_{2m-1}; \phi_{2m} \}$.

Then $E_m$ are invariant subspaces of $T+B$ and , (compare to \cite{admt2}, Lemma 14),
\begin{align*}
(T+B)_m = \frac{1}{2} ( t_{2m} + t_{2m-1} ) + \frac{1}{2} (t_{2m} - t_{2m-1}) 
\begin{bmatrix}-1 & s \\ -s & 1 \end{bmatrix}
\end{align*}
and 
\begin{align*}
(T+B)_m \psi_{m}^{\pm} = \left(\frac{1}{2}(t_{2m} + t_{2m-1}) \pm \frac{1}{2}(t_{2m}-t_{2m-1}) h \right), \quad h = 1/m,
\end{align*}
where 
\begin{align*}
\psi_m^{\pm} = g^{\pm} ( m) = \phi_{2m-1} + G^{\pm 1} \phi_{2m}.
\end{align*}

We omit further details. With the explicit formulas given it is easy to see that (\ref{0404_5}) with $h = 1/m$ 
guarantees that $\{ \psi_m^{\pm} \}_{1}^{\infty}$ is \emph{not} a basis.
\subsection{}
As an application of Theorem $5$, consider the differential operator $T$ on $L^2(\mathbb{R})$ defined by
\begin{align} \label{0801_1}
Ty = -y'' + |x|^\beta y, \quad \text{with} \quad \beta > 1.
\end{align}
The spectrum of $T$ consists of an infinite set of 
eigenvalues 
\begin{align*}
\text{Spec}T = \{ \lambda_0 \leq \lambda_1 \leq \lambda_2 \leq \ldots \} \quad 
\text{with} \quad \lim_{n\rightarrow \infty} \lambda_n = \infty.
\end{align*}
The growth of the sequence of eigenvalues is described by the formula
\begin{align} \label{2_26_1}
\lim_{n \rightarrow \infty} \left[ 2 \int_{0}^{\lambda_n^{1/\beta}} (\lambda_n - |x|^{\beta})^{1/2} dx - (n+1/2) \pi \right] = 0.
\end{align}
For a proof, see the last section of \cite{titchmarsh}.
It follows from (\ref{2_26_1}) by a change of variables that
\begin{align} \label{2_26_2}
\lim_{n \rightarrow \infty} \left[ 2 \lambda_n^{\frac{2+\beta}{2\beta}} \Omega_\beta - (n+1/2) \pi \right] = 0 \quad \text{with} 
\quad \Omega_\beta = 2 \int_{0}^{1} (1 - x^{\beta})^{1/2} dx. 
\end{align}
Subtracting the $n^{\text{th}}$ term from the $n+1^{\text{st}}$ term in (\ref{2_26_2}) we derive
\begin{align} \label{2_26_3}
\lim_{n \rightarrow \infty} \left[ \lambda_{n+1}^{\frac{2+\beta}{2 \beta}} - \lambda_n^{\frac{2+\beta}{2 \beta}} \right] =\pi/\Omega_\beta.
\end{align}
From (\ref{2_26_3}) it is straightforward to show that there exist constants $C > 0$, $N \in \mathbb{N}$ (depending on $\beta$) such that
\begin{align} \label{2_26_4}
\lambda_{n+1}- \lambda_n \geq C n^{\alpha - 1} \quad  \forall n > N, \quad \alpha = \frac{2\beta}{\beta+2}.
\end{align}
Let us mention the papers \cite{shin1}, \cite{shin2} where the eigenvalues for the eigenproblem 
$-y_{zz} + q(z)y =\lambda y$ are analyzed for polynomial $q(z)$.

Denote the eigenfunction corresponding to $\lambda_n$ by $\phi_n$
and define  
$$
L(p; \alpha) = \{ b: b(x) (1+|x|^2)^{-\alpha/2} \in L^p(\mathbb{R})\}.
$$
We have the following bound on the behavior of $\phi_n$
\begin{align} \label{2_26_5}
|\phi_n(x)| \leq  \frac{K \exp(Q(x))}{|\lambda_n - |x|^{\beta}|^{1/4} + \lambda_n^{\frac{\beta-1}{6\beta}}} \quad \text{with}\\ 
\notag
Q(x) = \begin{cases}
	- \int_{\lambda_n^{1/\beta}}^{x} (\lambda_n - |t|^{\beta})^{1/2} dt, \quad & x > \lambda_n^{1/\beta} \\
	0, & |x| \leq \lambda_n^{1/\beta} \\
	 \int_{\lambda_n^{1/\beta}}^{x} (\lambda_n - |t|^{\beta})^{1/2} dt, \quad & x < -\lambda_n^{1/\beta}. 
	\end{cases}
\end{align}
For the case $\beta = 2$, this inequality is proven in \cite{akhmerova}, a few changes to this proof 
boost it to cover $\beta > 1$. We omit the details. Such constructions for Schrodinger operators with
turning points are discussed in \cite[Ch 8,11]{olver}. 
 
By an argument like that given for Lemma $8$ in \cite{admt2} it follows from (\ref{2_26_5}) that if $b \in L(p;\alpha)$ then
$\|b \phi_n \|_2 \leq C n^{\frac{2 \beta \xi}{\beta + 2}}$ where
\begin{align} \label{2_26_6}
\xi &= \max \{ \frac{1}{3\beta}\left( 1 - \beta + 3 \alpha + (\beta - 1)/p\right);
\frac{1}{\beta} \left(\alpha - \beta/4 + 1/2 - 1/p \right) \} \\
\notag
&= \begin{cases}
	\frac{1}{3\beta}\left( 1 - \beta + 3 \alpha + (\beta - 1)/p\right), 
\quad &2 \leq p < 4 \\
	\frac{1}{\beta} \left(\alpha - \beta/4 + 1/2 - 1/p \right) \}, \quad &4 < p.	
   \end{cases}
\end{align}
In the exceptional case $p=4$ we have
\begin{align} \label{3_17_1}
\|b \phi_n \|_2 \leq C n^{\frac{2\alpha}{\beta+2} + \frac{1-\beta}{2(\beta+2)}} \log(n+2)
\end{align}
The following Proposition follows from 
(\ref{2_26_4}), (\ref{2_26_6}), (\ref{3_17_1}) and Theorem \ref{0704_11}. Our 
statement of Theorem \ref{0704_11} does not include $\alpha = 1$ (and therefore $\beta = 2$); a proper forumalation and proof 
of this Theorem for $\alpha = 1$ can be found in 
\cite{admt2}.
\begin{proposition}
Let $T \in (\ref{0801_1})$, $b \in L(p, \alpha)$, and define the operator $B$  on 
$L^2(\mathbb{R})$ by $Bf = b(x)f(x)$. \\
Suppose that  
\begin{align}
\begin{cases}
\beta - 1 <p ( -4 + 5 \beta/2 - 3\alpha)\quad \text{if} \quad  &2 \leq p < 4 \\
2 > p(3  - 3\beta/2+ 2 \alpha) \quad \text{if} \quad &4 \leq p.
\end{cases}
\end{align} 
Then the system of eigen and associated 
functions for the operator $T+B$ is an unconditional basis. 
\end{proposition} 
Acknowledgements: \quad 
We would like to thank P. Djakov, K. Shin, A. Shkalikov 
for interesting discussions related to this work.


\begin{thebibliography}{99}
\bibitem{admt2} J. Adduci, B. Mityagin, Eigensystem of an $L^2$-perturbed 
harmonic oscillator is an unconditional basis. arXiv 0912.2722.
\bibitem{akhmerova} E. Akhmerova, Spectral asymptotics for nonsmooth perturbations for the harmonic oscillator. Siberian Mathematical Journal, Vol. 49, No. 6,  968--984, 2008.
\bibitem{katopaper} T. Kato, Similarity for sequences of projections.
Bull. Amer. Math. Soc., 73, 1967, 904--905. 
\bibitem{katobook} T. Kato,  Perturbation theory for linear operators,
Springer Verlag, Berlin, 1980.
\bibitem{naimark} M. Naimark, Linear differential operators I, Frederic Ungar Publishing Co., New York, 1967.
\bibitem{olver} F. Olver, Asymptotics and special functions, Academic Press, New York, 1974.
\bibitem{shin1} K. Shin, Anharmonic oscillators with infinitely many real eigenvalues and PT-symmetry, SIGMA 6 (2010), 015, 9 pages, ArXiv: 1002.0798 
\bibitem{shin2} K. Shin, Anharmonic oscillators in the complex plane, PT-symmetry, and real eigenvalues. To appear in Potential analysis. arXiv 1008.0905.
\bibitem{shkalikov} A. Shkalikov, On basisness of root vectors of perturbed self-adjoint operator,
Trudy Math. Steklov Institute, 269 (2010), 1-15.
\bibitem{titchmarsh}	E. Titchmarsh, On the asymptotic distribution of eigenvalues. Quarterly journal of mathematics, 5, 1954, 228--240.
\end{thebibliography}
\end{document}